\documentclass[12pt]{amsart}

\usepackage[margin=2.3cm]{geometry}
\usepackage[utf8]{inputenc}
\usepackage{amsmath,amssymb,amsthm,mathtools}
\usepackage[colorlinks=true,linkcolor=blue,citecolor=blue,urlcolor=blue]{hyperref}

\newtheorem{theorem}{Theorem}
\newtheorem{lemma}{Lemma}

\theoremstyle{definition}

\newtheorem{example}{Example}
\theoremstyle{remark}

\newcommand{\R}{\mathbb{R}}

\newcommand{\Z}{\mathbb{Z}}
\newcommand{\N}{\mathbb{N}}

\DeclareMathOperator{\dist}{dist}
\DeclareMathOperator{\Id}{Id}
\newcommand{\nn}[1]{\left\|#1\right\|}
\allowdisplaybreaks[1]

\title[Continued fractions entropy and base $b$ entropy]{Simultaneously small continued fraction entropy and base $b$ entropy}
\author{Faruk Temur}
\address{Department of Mathematics\\
Izmir Institute of Technology, Urla, Izmir, 35430, Turkey}
\email{faruktemur@iyte.edu.tr}
\subjclass[2020]{Primary 11A55, 11K16; Secondary 37B10}
\keywords{Continued fractions, block complexity, entropy, rich numbers}
\date{September 14, 2026}

\begin{document}

\begin{abstract}
We construct a computable irrational number that simultaneously has  small continued fraction entropy  and base $b$ entropy for every integer base $b\geq 2.$ This addresses a question raised by Bugeaud.
\end{abstract}

\maketitle

\section{Introduction}

A real number can be represented both by an integer base expansion and
by a continued fraction. The number of distinct finite blocks occurring
in either representation gives a measure of its combinatorial complexity. Yann Bugeaud, in his now classical book \cite{Bugeaud2012},  raised the question of simultaneously controlling these two types of complexity.

Let $b\geq2$ be an integer, and let $\xi\in(0,1)$ be irrational. The base $b$ expansion of $\xi$ is given by
$\xi=0.a_1a_2a_3\ldots$ with $0\leq a_j\leq b-1$.
The block complexity of this expansion is
\begin{equation*}\label{eq:base-complexity}
 p(n,\xi,b)=\#\{a_ka_{k+1}\cdots a_{k+n-1}:k\geq1\},
 \qquad n\geq1.
\end{equation*}
The number $\xi$ is called \emph{rich} to base $b$ if
$p(n,\xi,b)=b^n$ for every $n\geq1$. Thus richness means that every finite
word over the digit alphabet occurs in the expansion. The entropy to base
$b$ is
\begin{equation*}\label{eq:base-entropy}
 E(\xi,b)=\lim_{n\to\infty}\frac{\log p(n,\xi,b)}{n}.
\end{equation*}
The limit exists because the sequence $\log p(n,\xi,b)$ is subadditive.

For the continued fraction $\xi=[0;c_1,c_2,\ldots]$, we similarly put
\begin{equation}\label{eq:cf-entropy}
 p(n,\xi)=\#\{c_kc_{k+1}\cdots c_{k+n-1}:k\geq1\},
 \qquad
 E(\xi)=\lim_{n\to\infty}\frac{\log p(n,\xi)}{n}.
\end{equation}
All continued fractions constructed below have bounded partial quotients,
so the quantities in \eqref{eq:cf-entropy} are nonnegative real numbers.

Bugeaud  formulates the following assertion in  Chapter 10 of his book \cite{Bugeaud2012}.

\medskip
\noindent\textbf{Problem 10.53.}
\emph{There is an irrational real number $\xi$ such that $E(\xi)<\log2$ and
$E(\xi,b)<\log b$ for some integer $b\geq2$ (respectively, for every integer
$b\geq2$).}
\medskip

That uncountably many such numbers exist follows by combining a result of  Broderick et al. \cite{BBFKW2010} with 
 the  theory of cookie-cutter sets.  Corollary 1.2 of \cite{BBFKW2010} states that if a set $K$ supports a measure $\mu$ with a power law, that is, there exist $k_1,k_2,\gamma,\rho_0>0$ such that for all $x\in \text{supp}\mu$, and $0<\rho<\rho_0$ we have
 \begin{equation*}
 k_1 \rho^{\gamma}\leq	\mu(B(x,\rho))\leq k_2\rho^{\gamma},
 \end{equation*}
 then the set of points in $K$ that are badly approximable and such that,
 for every $b \geq 2$, their base $b$ expansion does not contain more than $C(x, b)$
 consecutive identical digits has the same Hausdorff dimension as $K$. If we choose 
 \begin{equation*}
 	K:=\{[0;1,a_1,1,a_2,1,a_3,\ldots]: a_j\in \{1,2\}\},
 \end{equation*}
 the elements $\xi$ of this set plainly satisfy $E(\xi)\leq (\log 2)/2$. This is a Gauss-Cantor set, or a cookie-cutter set, and it is the attractor of the  iterated function system generated by the  two maps
 \begin{equation}\label{eqT}
 	f_1=T_1\circ T_1, \qquad \qquad  f_2=T_1\circ T_2,   \qquad  \text{where}  \qquad T_a(t)=\frac{1}{a+t},
 \end{equation} 
 as applying $T_a$ to $\xi$ introduces $a$ to the beginning of the  continued fraction expansion of $\xi.$ It is well known that a cookie-cutter set has a positive Hausdorff dimension $s$,  its $s$-dimensional  Hausdorff  measure is positive and finite, 
 and restricted to it, satisfies a power law with the  power $s$, see for example the classic book \cite{kf} Theorem 5.3, or \cite{bedford,mu,akt}. In particular numerical evaluation of the Bowen formula gives that our $K$ has Hausdorff dimension $s\approx0.3057$. Therefore the set of points  $\xi\in K$ which are badly approximable and satisfying $E(\xi,b)<\log b$  have  the same Hausdorff dimension $s\approx0.3057$ by \cite{BBFKW2010}.

The purpose of this note is to  construct a computable  number satisfying the requirements of Bugeaud's problem. 
 Our main result is the following. 

\begin{theorem}\label{thm:all-bases}
There is a computable irrational number
\begin{equation}\label{eq:cf-pattern}
 \xi=[0;1,1+\varepsilon_1,1,1,\,
          1,1+\varepsilon_2,1,1,\,
          1,1+\varepsilon_3,1,1,\ldots],
 \qquad \varepsilon_j\in\{0,1\},
\end{equation}
such that, for every integer $b\geq2$ and every $k\geq0$,
\begin{equation}\label{eq:all-base-bound}
 \nn{b^k\xi}>
 \delta_b:=154^{-2^{b}}.
\end{equation}
In particular,
\begin{equation}\label{eq:main-entropies}
 E(\xi)\leq\frac{\log2}{4}<\log2,
 \qquad E(\xi,b)<\log b\quad(b\geq2).
\end{equation}
After $N$ stages, the construction gives a rational interval containing
$\xi$ and having length at most $20^{-N}$.
\end{theorem}

The construction  at each step chooses one of two intervals, indexed by $\varepsilon_{j}\in\{0,1\}$, of rational endpoints  so as to avoid a grid   $b^{-k}\Z$. The need to handle all bases simultaneously forces us to choose a rather sophisticated schedule for dealing with bases $b$, in which every base  appears infinitely often and with required frequency. 

The next section introduces the notation and preliminaries. The third and final section proves the theorem.

\section{Notation and Preliminaries}

In what follows we will use the following conventions. We let $\Z_+$ denote the nonnegative integers.
For a closed interval $I$, its length is denoted by $|I|$. For nonempty
sets $A,B\subset\R$, we write
\begin{equation*}
	\dist(A,B)=\inf\{|x-y|:x\in A,\ y\in B\}.
\end{equation*}
A real number is computable if it admits rational approximations to any
prescribed positive rational accuracy by a finite algorithm.

We start with a lemma to recall the relations between entropy to base $b$, missing words, and distance to the nearest integer.

\begin{lemma}\label{lem:missing-word}
Let $b\geq2$, and suppose that the base-$b$ expansion of an irrational
number $\xi\in(0,1)$ omits a word of length $\ell\geq1$. Then
\begin{equation}\label{eq:missing-word-entropy}
 E(\xi,b)\leq\frac{\log(b^\ell-1)}{\ell}<\log b.
\end{equation}
In particular, the word $0^\ell$ is omitted if
$\nn{b^k\xi}>\delta$ for every $k\geq0$ and $b^{-\ell}\leq\delta$.
\end{lemma}

\begin{proof}
Every occurring word of length $m\ell$ splits into $m$ words of length
$\ell$, none of which is the omitted word. Hence
\begin{equation*}
 p(m\ell,\xi,b)\leq(b^\ell-1)^m.
\end{equation*}
Taking logarithms, dividing by $m\ell$, and letting $m$ tend to infinity
proves \eqref{eq:missing-word-entropy}. For the last assertion, a block of $\ell$ zeros beginning at
position $k+1$ would give $\{b^k\xi\}\leq b^{-\ell}\leq\delta$, contrary
to $\{b^k\xi\}\geq\nn{b^k\xi}>\delta$.
\end{proof}

We now recall some basic concepts and facts on continued fractions, for further details consult \cite{Bugeaud2012}.  We call   $[0;a_1,a_2,\ldots, a_h,\ldots]$ continued fractions with prefix $(a_1,a_2,\ldots a_h)$. The set of all continued fractions with a certain prefix is called the cylinder of that prefix. We recall the maps $T_a$ defined in \eqref{eqT}. For $a\in \N$, these maps are smooth injections of the interval $[0,1]$ into itself, but they reverse orientation, as  their first derivatives are strictly negative. Below we will work with words of even length, and therefore will compose an even number of these maps, and thus our maps will always have strictly positive first derivative. 
 For a finite word of  positive integers
$W=(a_1,\ldots,a_h)$ with even length, we define
\begin{equation*}
 F_W=T_{a_1}\circ\cdots\circ T_{a_h},\qquad \qquad \qquad
 I_W=F_W([0,1]).
\end{equation*}
We also allow the empty word, with $F_{\varnothing}=\Id$ and
$I_{\varnothing}=[0,1]$. Write $p_j/q_j=[0;a_1,\ldots,a_j]$, with the
initial values $p_{-1}=1$, $p_0=0$, $q_{-1}=0$, $q_0=1$. The usual
continued fraction recursions give
\begin{equation}\label{eq:cf-map}
 F_W(t)=\frac{p_h+p_{h-1}t}{q_h+q_{h-1}t},
 \qquad  \qquad
 |I_W|=\frac1{q_h(q_h+q_{h-1})}.
\end{equation}
These identities also hold for the empty word.  We have $0\leq q_{h-1}\leq q_h$, and
\begin{equation}\label{eq:cf-distortion}
 \frac{F_W'(t)}{|I_W|}
 =\frac{q_h(q_h+q_{h-1})}{(q_h+q_{h-1}t)^2},
 \qquad \qquad
 \frac {|I_W|}{2}\leq F_W'(t)\leq2|I_W|
 \quad(0\leq t\leq1).
\end{equation}
Thus the distortion estimate is independent of the length of the prefix.

Consider the two four-letter words
\begin{equation*}
 W_0=(1,1,1,1),\qquad \qquad   \qquad  W_1=(1,2,1,1).
\end{equation*}
Their maps and  closures of their cylinders are
\begin{equation*}\label{eq:two-cf-maps}
 \begin{aligned}
 \varphi_0(t)&=\frac{2t+3}{3t+5},  \qquad \qquad
 &A_0&=\varphi_0([0,1])=\left[\frac35,\frac58\right],\\
 \varphi_1(t)&=\frac{3t+5}{4t+7},
 &A_1&=\varphi_1([0,1])=\left[\frac57,\frac8{11}\right].
 \end{aligned}
\end{equation*}
In particular,
\begin{equation}\label{eq:cf-first-geometry}
 |A_0|=\frac1{40},\qquad |A_1|=\frac1{77},\qquad
 \dist(A_0,A_1)=\frac5{56}.
\end{equation}

\begin{lemma}\label{lem:cf-avoidance}
Let $F$ correspond to an even length continued fraction prefix, including
the empty prefix, and put $I=F([0,1])$. For $i=0,1$, let
$J_i=F(A_i)$. Then
\begin{equation}\label{eq:cf-contraction}
 \frac {|I|}{154}\leq |J_i|\leq\frac{|I|}{20}.
\end{equation}
If $G$ is a grid with spacing at least $|I|$, at least one of the two
children  $J_0,J_1$ satisfies
\begin{equation}\label{eq:cf-avoidance}
 \dist(J_i,G)>\frac {|I|}{128}.
\end{equation}
\end{lemma}

\begin{proof}
By merely integrating \eqref{eq:cf-distortion} over the intervals in question   we obtain \eqref{eq:cf-contraction}, and also the  following. The distance between the children is at
least $5|I|/112$. Their convex hull has distance at least $3|I|/10$ from the
left endpoint of $I$, and at least $3|I|/22$ from its right endpoint.
Enlarging each child by $|I|/128$ on both sides therefore gives two disjoint
closed intervals whose convex hull lies strictly inside $I$. That hull
has length less than $|I|$, so it contains at most one point of $G$.
Consequently at least one enlarged child misses $G$, which proves
\eqref{eq:cf-avoidance}.
\end{proof}

We also record how the grid-avoidance condition can be checked exactly.
For rational $u\leq v$, rational $d>0$, and an integer $Q\geq1$, the
condition $\dist([u,v],Q^{-1}\Z)>d$ says that the closed interval
$[Q(u-d),Q(v+d)]$ contains no integer. It is therefore equivalent to
\begin{equation}\label{eq:grid-test}
	\left\lceil Q(u-d)\right\rceil
	>\left\lfloor Q(v+d)\right\rfloor.
\end{equation}
In particular, it is a finite rational comparison.

\section{Proof of the main theorem}

We now prove Theorem 1. Any $\xi$ of the form in \eqref{eq:cf-pattern} satisfies the desired conclusion on continued fraction entropy $E(\xi).$ Hence, our main  aim is to choose $\varepsilon_j$ so that the resulting $\xi$ satisfies \eqref{eq:all-base-bound}.  At each $j$, we  choose $\varepsilon_j$ so that $\xi$ will lie in an interval that avoids the grid $b^{-k}\Z$ for some base $b\geq 2$ and some power $k\in \N$. These intervals  will of course be nested, and their intersection is the desired $\xi$. The most crucial part of the following construction is to decide which base $b$ and which power $k$ to handle at a given step $j.$ In \eqref{eq:all-base-bound} we have information on all powers of all bases, and  moreover the right hand side is independent of $k$. To get this information any base $b$  appears at infinitely many steps, and with sufficient frequency. Once we schedule which base to consider at each step $j$ appropriately, Lemma 2 will allow us choose an appropriate subinterval to avoid  powers of this base up to some level. At this interval choosing step, our intervals have rational endpoints, and we are avoiding a grid of the form  $b^{-k}\Z$. So each step reduces to comparing two rational numbers, and therefore is a  finite computation.

\begin{proof}[Proof of Theorem~\ref{thm:all-bases}]
We choose the base $b$ to handle at step $j$ according to the following schedule. For $j\geq1$, let $v_j$ be the largest integer  such that
$2^{v_j}$ divides $j$. At stage $j$,  we address the base
$
 b_j=2+v_j.
$
Thus the schedule begins $2,3,2,4,2,3,2,5,\ldots$ For a fixed base
$b\geq2$, its assigned stages form the progression
\begin{equation*}\label{eq:assigned-stages}
 \tau_b,\ \tau_b+D_b,\ \tau_b+2D_b,\ldots,
 \qquad \tau_b=2^{b-2},\quad D_b=2^{b-1}.
\end{equation*}
The bounded gap for each fixed base is what makes \eqref{eq:all-base-bound} independent of $k$.

We will  construct a sequence of maps $F_j,  j\in \Z_+$ on the interval $[0,1]$.   As the base  step, we take $F_0=\Id.$  The 
 step  $j\in \N$ of this construction will determine $\varepsilon_{j}$, we will define 
 $F_j=F_{j-1}\circ \varphi_{\varepsilon_j}$, and continue. So we will indeed have  $F_j=\varphi_{\varepsilon_1}\circ\cdots\circ \varphi_{\varepsilon_j}$.

  Given $F_{j-1}$, let
$I_{j-1}=F_{j-1}([0,1])$ and $L_{j-1}=|I_{j-1}|$, and set
\begin{equation*}\label{eq:all-base-grid}
 Q_j=\max\{b_j^k:k\in\Z_+,\ b_j^k\leq L_{j-1}^{-1}\},
 \qquad G_j=Q_j^{-1}\Z.
\end{equation*}
The spacing of $G_j$ is at least $L_{j-1}$. Define $\varepsilon_j$ to be
the smallest $i\in\{0,1\}$ such that
\begin{equation}\label{eq:cf-choice}
 \dist(F_{j-1}(A_i),G_j)>\frac{L_{j-1}}{128},
\end{equation}
and put
\begin{equation*}\label{eq:cf-recursion}
 F_j=F_{j-1}\circ\varphi_{\varepsilon_j},
 \qquad I_j=F_j([0,1]),\qquad L_j=|I_j|.
\end{equation*}
Lemma~\ref{lem:cf-avoidance} shows that this choice always exists and that
\begin{equation}\label{eq:all-lengths}
 \frac{L_{j-1}}{154}\leq L_j\leq\frac{L_{j-1}}{20},
 \qquad 154^{-j}\leq L_j\leq20^{-j}.
\end{equation}
The nested intervals have a unique common point $\xi$. Their prefixes
show that its continued fraction is \eqref{eq:cf-pattern}. As its continued fraction expansion is infinite,  $\xi$ is irrational.

Fix $b\geq2$ and $q=b^k$, where $k\geq0$. Let $j$ be the first stage
assigned to $b$ at which
\begin{equation}\label{eq:first-covering-stage}
 q\leq L_{j-1}^{-1}.
\end{equation}
Such a stage exists by \eqref{eq:all-lengths}. Since $Q_j$ is a power of
$b$ at least as large as $q$, the inclusion $q^{-1}\Z\subseteq G_j$ and
\eqref{eq:cf-choice} give
\begin{equation}\label{eq:all-scale-bound}
 \nn{q\xi}>\frac{qL_{j-1}}{128}.
\end{equation}

If $j>\tau_b$, the preceding assigned stage is $j-D_b$. Minimality in
\eqref{eq:first-covering-stage} gives
$q>L_{j-D_b-1}^{-1}$. Applying the lower contraction bound in
\eqref{eq:all-lengths} over $D_b$ stages yields
\begin{equation*}
 qL_{j-1}>
 {L_{j-1}}/{L_{j-D_b-1}}\geq154^{-D_b}.
\end{equation*}
If $j=\tau_b$, then $q\geq1$ and
\begin{equation*}
 qL_{j-1}\geq L_{j-1}\geq154^{-(\tau_b-1)}\geq154^{-D_b}.
\end{equation*}
Together with \eqref{eq:all-scale-bound}, this proves
\eqref{eq:all-base-bound}, including the case $k=0$.

For each $b\geq2$, define the explicit integer
\begin{equation}\label{eq:missing-length}
 \ell_b=\min\{\ell\in\N:b^\ell\geq154^{\,2^{b}}\}.
\end{equation}
Then $b^{-\ell_b}\leq\delta_b$. By Lemma~\ref{lem:missing-word}, the
base $b$ expansion omits $0^{\ell_b}$ and
\begin{equation}\label{eq:all-base-entropy}
 E(\xi,b)\leq\frac{\log(b^{\ell_b}-1)}{\ell_b}<\log b.
\end{equation}
This proves \eqref{eq:main-entropies}.

Finally, every $F_j$ has integer coefficients and rational endpoint
values. The integer $Q_j$ is found by successively multiplying by $b_j$,
and the test in \eqref{eq:cf-choice} is the rational test
\eqref{eq:grid-test}. The choice of $\varepsilon_j$ is therefore effective.
The rational intervals $I_N$ have length at most $20^{-N}$ and contain
$\xi$, giving the asserted approximation bound.
\end{proof}

\begin{example}
Running the algorithm on a computer,   the first six choices $\varepsilon_j$ are zero and
$\varepsilon_7=1$. Our $\xi$ begins  $\xi=0.618033988759272844854423252344\ldots.$
\end{example}

\end{document}